\documentclass[11pt]{article}

\usepackage[margin=1in]{geometry}
\usepackage[T1]{fontenc}
\usepackage[utf8]{inputenc}
\usepackage{amsmath,amssymb,amsthm,mathtools}
\usepackage{enumitem}
\usepackage{hyperref}

\hypersetup{
    colorlinks=true,
    linkcolor=blue,
    citecolor=blue,
    urlcolor=blue
}

\newtheorem{theorem}{Theorem}
\newtheorem{lemma}[theorem]{Lemma}
\newtheorem{corollary}[theorem]{Corollary}
\theoremstyle{definition}
\newtheorem{definition}[theorem]{Definition}
\theoremstyle{remark}

\newcommand{\R}{\mathbb{R}}
\newcommand{\Z}{\mathbb{Z}}
\newcommand{\dd}{\mathrm{d}}
\newcommand{\Div}{\operatorname{div}}

\title{Finite Three-Colourable $(0,2)$-Graphs Are Bipartite}
\author{Christopher Williamson}
\date{\today}

\begin{document}
\maketitle

\begin{abstract}
A theorem of Payan says that a cubelike graph cannot have chromatic number exactly three.  A nearby question, usually discussed as Payan's finite $(0,2)$-graph question, asks whether a finite graph in which every two distinct vertices have either zero or two common neighbours can have chromatic number exactly three.  The finite hypothesis is meaningful: infinite three-chromatic $(0,2)$-graphs can be constructed \cite{Payan1992}.

We prove that every finite three-colourable $(0,2)$-graph is bipartite.  Thus, no finite $(0,2)$-graph has chromatic number exactly three.
\end{abstract}

\section{Background and statement}

All graphs in this note are finite, simple, and undirected unless explicitly stated otherwise.  For a vertex $v$ of a graph $G$, write $N(v)$ for its neighbourhood.

\begin{definition}
A graph $G$ is a \emph{$(0,2)$-graph} if, for all distinct vertices $x,y$,
\[
    |N(x)\cap N(y)|\in\{0,2\}.
\]
\end{definition}

Payan proved that no cubelike graph has chromatic number exactly three \cite{Payan1992} and asked whether the same holds for finite $(0,2)$-graphs. This problem is also recorded in the graph theory open problems compilation \cite{GraphConjectures2026}.  Payan notes that infinite $(0,2)$-graphs of chromatic number three can be constructed.

The purpose of this note is to prove the following theorem.

\begin{theorem}\label{thm:main}
If $G$ is a finite three-colourable $(0,2)$-graph, then $G$ is bipartite.  Consequently, no finite $(0,2)$-graph has chromatic number exactly three.
\end{theorem}

We use colours in $\Z/3\Z$.  Thus a proper three-colouring is a map
\[
    c:V(G)\to \Z/3\Z
\]
with $c(x)\ne c(y)$ whenever $x$ and $y$ are adjacent.

\section{The triangle-free reduction}

Given a proper three-colouring $c$, orient every edge cyclically: orient $u v$ as $u\to v$ if
\[
    c(v)=c(u)+1\pmod 3.
\]
For a vertex $v$, let $p(v)$ be the number of incoming edges at $v$ in this orientation.

\begin{lemma}\label{lem:p-increases}
Let $G$ be a finite $(0,2)$-graph with a proper three-colouring and the induced cyclic orientation.  If $u\to v$, then
\[
    p(v)\ge p(u)+1.
\]
\end{lemma}

\begin{proof}
Let $c(u)=i$, so $c(v)=i+1$.  For every predecessor $x$ of $u$, we have $c(x)=i-1$.  The vertices $x$ and $v$ have $u$ as a common neighbour, so by the $(0,2)$ condition they have exactly one further common neighbour; call it $y$.

Since $y$ is adjacent to $x$, its colour is either $i$ or $i+1$.  Since $y$ is adjacent to $v$, its colour is either $i-1$ or $i$.  Hence $c(y)=i$, so $y\to v$.

This sends each predecessor $x$ of $u$ to a predecessor $y$ of $v$.  The map is injective: if distinct predecessors $x_1,x_2$ of $u$ were sent to the same $y$, then $u$ and $y$ would have at least three common neighbours, namely $x_1,x_2,v$, contradicting the $(0,2)$ condition.  Finally, $u$ itself is also a predecessor of $v$, and it is not in the image of the map.  Thus $v$ has at least $p(u)+1$ predecessors.
\end{proof}

\begin{corollary}\label{cor:triangle-free}
A finite three-colourable $(0,2)$-graph is triangle-free.
\end{corollary}

\begin{proof}
In a properly three-coloured triangle, the three colours are $0,1,2$, so the cyclic orientation makes the triangle a directed cycle.  Lemma~\ref{lem:p-increases} would then give a strict increase of $p$ around a closed directed cycle, which is impossible.
\end{proof}

We shall also need the following immediate consequence.

\begin{lemma}\label{lem:rectagraph-property}
Let $G$ be a triangle-free $(0,2)$-graph.  Then every path of length two lies in a unique $4$-cycle.  More precisely, if $a-u-v$ is a path of length two, then there is a unique vertex $b$ such that
\[
    a-u-v-b-a
\]
is a $4$-cycle.
\end{lemma}

\begin{proof}
The vertices $a$ and $v$ have $u$ as a common neighbour, so by the $(0,2)$ condition they have exactly two common neighbours.  Let $b$ be the common neighbour different from $u$.  Then $b$ is adjacent to both $a$ and $v$.  Triangle-freeness ensures that $b$ is distinct from $a$ and $v$, so $a-u-v-b-a$ is a genuine $4$-cycle.  Uniqueness follows because $a$ and $v$ have only two common neighbours.
\end{proof}

Thus, after the triangle-free reduction, the graph is a rectagraph in the usual sense: no triangles, and each length-two path is contained in a unique square.

\section{A least-squares lemma for square-closed edge labels}

Let $G$ be a graph.  A real edge-label is a function $\omega$ on oriented edges such that
\[
    \omega(y,x)=-\omega(x,y)
\]
for every edge $xy$.  Its norm is
\[
    \|\omega\|^2:=\sum_{\{x,y\}\in E(G)} \omega(x,y)^2,
\]
where either orientation may be chosen for each edge.  This is well-defined because reversing an orientation changes a label by a sign but does not change its square.

If $f:V(G)\to\R$, define the gradient edge-label
\[
    \dd f(x,y):=f(y)-f(x).
\]
For an edge-label $\omega$, define its divergence at $x$ by
\[
    (\Div\omega)(x):=\sum_{y\sim x}\omega(x,y).
\]

\begin{lemma}\label{lem:least-squares}
Let $G$ be a finite triangle-free $(0,2)$-graph.  Suppose $\alpha$ is a real edge-label whose sum around every $4$-cycle is zero.  Then there is a function $F:V(G)\to\R$ such that
\[
    \alpha=\dd F.
\]
\end{lemma}

\begin{proof}
Consider all edge-labels of the form
\[
    \alpha+\dd f,
    \qquad f:V(G)\to\R.
\]
We first justify that there is one of minimum norm.

Choose one arbitrary orientation for each undirected edge of $G$.  If $G$ has $m$ edges, then a real edge-label is just a vector in $\R^m$: one real number for each chosen oriented edge.  Thus the space of all real edge-labels is a finite-dimensional Euclidean space.  The vertex functions $f:V(G)\to\R$ form another finite-dimensional vector space, and the map
\[
    f\longmapsto \dd f
\]
is linear.  Therefore the set of all gradients
\[
    B:=\{\dd f:f:V(G)\to\R\}
\]
is a linear subspace of the edge-label space.  In finite-dimensional Euclidean space every linear subspace is closed.  Hence
\[
    \alpha+B=\{\alpha+\dd f:f:V(G)\to\R\}
\]
is a closed affine subspace.

Since $\alpha$ itself lies in $\alpha+B$, the smallest possible norm is at most $\|\alpha\|$.  Thus no vector of norm greater than $\|\alpha\|$ can be needed to find the minimum.  It is enough to minimise the norm on
\[
    (\alpha+B)\cap\{\gamma:\|\gamma\|\le \|\alpha\|\}.
\]
This set is closed and bounded in a finite-dimensional Euclidean space, hence compact.  The norm is continuous, so the extreme value theorem gives a minimiser.  Choose such a minimum-norm label and call it $\omega$.  Thus
\[
    \omega=\alpha+\dd f_0
\]
for some vertex function $f_0$, and
\[
    \|\omega\|\le \|\alpha+\dd f\|
    \qquad\text{for every }f:V(G)\to\R.
\]

The label $\omega$ is still square-closed: $\alpha$ has sum zero around every $4$-cycle by hypothesis, and every gradient has sum zero around every cycle by telescoping.

We next prove, directly from minimality, that
\begin{equation}\label{eq:div-free}
    (\Div\omega)(x)=0
    \qquad\text{for every }x\in V(G).
\end{equation}
Fix a vertex $x$, and let $\mathbf{1}_x$ be the function that is $1$ at $x$ and $0$ at every other vertex.  For every real number $t$, the label
\[
    \omega+t\,\dd\mathbf{1}_x
\]
still has the form $\alpha+\dd f$, so it lies in the same affine family over which $\omega$ minimises the norm.  Hence the one-variable function
\[
    \Phi(t):=\|\omega+t\,\dd\mathbf{1}_x\|^2
\]
has a minimum at $t=0$.

Let us expand $\Phi(t)$ without using inner-product notation.  If $y\sim x$, then
\[
    \dd\mathbf{1}_x(x,y)=\mathbf{1}_x(y)-\mathbf{1}_x(x)=0-1=-1.
\]
For edges not incident with $x$, the value of $\dd\mathbf{1}_x$ is zero.  Therefore
\[
    \Phi(t)
    =
    \sum_{y\sim x}\bigl(\omega(x,y)-t\bigr)^2
    +
    C,
\]
where $C$ is the contribution of the edges not incident with $x$ and is independent of $t$.  Differentiating this quadratic gives
\[
    \Phi'(t)
    =
    -2\sum_{y\sim x}\bigl(\omega(x,y)-t\bigr),
\]
and hence
\[
    \Phi'(0)
    =
    -2\sum_{y\sim x}\omega(x,y)
    =
    -2(\Div\omega)(x).
\]
Since $\Phi$ has a minimum at $t=0$, we have $\Phi'(0)=0$.  Thus $(\Div\omega)(x)=0$.  Since $x$ was arbitrary, \eqref{eq:div-free} follows.

We now show that $\omega=0$.  Suppose not.  Choose an oriented edge $(u,v)$ with
\[
    \omega(u,v)=W>0
\]
and with $W$ maximal among all absolute edge-labels.  If $u$ has degree $1$, then \eqref{eq:div-free} at $u$ gives $W=0$, a contradiction.  Hence $u$ has degree at least $2$.

By Lemma~\ref{lem:rectagraph-property}, for every $a\in N(u)\setminus\{v\}$ there is a unique vertex $b(a)\in N(v)\setminus\{u\}$ such that
\[
    a-u-v-b(a)-a
\]
is a $4$-cycle.  The same uniqueness applied in the reverse direction shows that
\[
    a\mapsto b(a)
\]
is a bijection from $N(u)\setminus\{v\}$ to $N(v)\setminus\{u\}$.  In particular $u$ and $v$ have the same degree; call it $d$.

For each $a\in N(u)\setminus\{v\}$, the zero sum of $\omega$ around the square $u-v-b(a)-a-u$ gives
\[
    \omega(u,v)+\omega(v,b(a))+\omega(b(a),a)+\omega(a,u)=0.
\]
Equivalently,
\[
    \omega(a,b(a))
    = W+\omega(v,b(a))+\omega(a,u).
\]
Summing this identity over all $a\in N(u)\setminus\{v\}$ gives
\begin{align*}
    \sum_a \omega(a,b(a))
    &= (d-1)W
       +\sum_{b\in N(v)\setminus\{u\}}\omega(v,b)
       +\sum_{a\in N(u)\setminus\{v\}}\omega(a,u).
\end{align*}
Using $(\Div\omega)(v)=0$ gives
\[
    \sum_{b\in N(v)\setminus\{u\}}\omega(v,b)=W,
\]
and using $(\Div\omega)(u)=0$ gives
\[
    \sum_{a\in N(u)\setminus\{v\}}\omega(a,u)=W.
\]
Therefore
\begin{equation}\label{eq:too-large}
    \sum_a \omega(a,b(a))=(d+1)W.
\end{equation}
But the sum on the left of \eqref{eq:too-large} has only $d-1$ terms, and each term has absolute value at most $W$ by the choice of $W$.  Hence its absolute value is at most $(d-1)W$, contradicting \eqref{eq:too-large}.  Thus $\omega=0$.

Since $\omega=\alpha+\dd f_0$ for some $f_0$, we have
\[
    \alpha=\dd(-f_0).
\]
This proves the lemma.
\end{proof}

\section{The colouring label}

Let $G$ be a finite three-colourable $(0,2)$-graph, and fix a proper colouring $c:V(G)\to\Z/3\Z$.  By Corollary~\ref{cor:triangle-free}, $G$ is triangle-free.

Define a real edge-label $\eta$ by
\[
    \eta(x,y)=
    \begin{cases}
        1, & c(y)=c(x)+1\pmod 3,\\
       -1, & c(y)=c(x)-1\pmod 3.
    \end{cases}
\]
This is well-defined because adjacent vertices have distinct colours.

\begin{lemma}\label{lem:eta-square-closed}
The label $\eta$ has sum zero around every $4$-cycle.
\end{lemma}

\begin{proof}
Let $v_0v_1v_2v_3v_0$ be a $4$-cycle.  For each $i$ modulo $4$, we claim that
\[
    \eta(v_i,v_{i+1})\equiv c(v_{i+1})-c(v_i)\pmod 3.
\]
This is just the two cases in the definition of $\eta$.  Since the colouring is proper, the colour difference across an edge is either $+1$ or $-1$ in $\Z/3\Z$.  If the difference is $+1$, then $\eta$ is the integer $1$, which represents the same residue modulo $3$.  If the difference is $-1$, then $\eta$ is the integer $-1$, again representing the same residue modulo $3$.  No other case can occur.

It follows that
\[
    \sum_{i=0}^3 \eta(v_i,v_{i+1})
    \equiv
    \sum_{i=0}^3 (c(v_{i+1})-c(v_i))
    \equiv 0
    \pmod 3,
\]
because the colour differences telescope around the closed walk.  The same sum is a sum of four numbers, each equal to $1$ or $-1$, so it lies in
\[
    \{-4,-2,0,2,4\}.
\]
The only number in this set congruent to $0$ modulo $3$ is $0$.  Hence the sum of $\eta$ around the $4$-cycle is zero.
\end{proof}

By Lemma~\ref{lem:least-squares} applied to $\eta$, there is a function $F:V(G)\to\R$ such that
\begin{equation}\label{eq:eta-gradient}
    \eta=\dd F.
\end{equation}

\section{Completion of the proof}

Let $C=v_0v_1\cdots v_{\ell-1}v_0$ be any cycle of $G$.  Summing \eqref{eq:eta-gradient} around $C$ gives
\[
    \sum_{i=0}^{\ell-1}\eta(v_i,v_{i+1})
    =
    \sum_{i=0}^{\ell-1}(F(v_{i+1})-F(v_i))
    =0.
\]
However, each summand is $1$ or $-1$.  If $\ell$ were odd, the sum of these $\ell$ signs could not be zero.  Thus $G$ has no odd cycle.

A finite graph with no odd cycle is bipartite.  Therefore $G$ is bipartite.  This proves Theorem~\ref{thm:main}.

\begin{corollary}
There is no finite $(0,2)$-graph with chromatic number exactly three.
\end{corollary}

\begin{proof}
If such a graph existed, it would be three-colourable, hence bipartite by Theorem~\ref{thm:main}.  But bipartite graphs have chromatic number at most two.
\end{proof}

\section*{Acknowledgements}

The author thanks the maintainers of the graph-theory open-problems
compilation at \url{https://mlelarge.github.io/graph-conjectures/} and
the Open Problem Garden for making this problem and its surrounding
references easy to locate.

The author also acknowledges substantial assistance from OpenAI's
ChatGPT in investigating the problem, checking the argument, simplifying
the proof, and preparing the exposition.  Responsibility for the final
mathematical content, including any errors, remains with the author.

\end{document}